\documentclass[12pt, reqno]{amsart}
\usepackage{longtable}
\usepackage[dvipsnames]{xcolor}
\usepackage{tikz}
\usepackage{stmaryrd}
\usetikzlibrary{matrix,arrows,decorations.pathmorphing}
\usepackage{amsfonts}
\usepackage{pdfpages}
\usepackage{graphicx}
\usepackage{mathpazo}
\usepackage{euler}
\usepackage{amssymb}
\usepackage{amsmath,mathtools}
\usepackage{fullpage}
\usepackage{caption}
\usepackage{amsthm}
\usepackage{thmtools}
\usepackage{thm-restate}
\usepackage{enumerate}
\usepackage{picinpar} 
\usepackage{tikz-cd}
\usepackage{todonotes}
\usepackage{color}
\usepackage{latexsym}
\usepackage[toc,page]{appendix}
\usepackage{url} 
\usepackage{multicol}
\usepackage{rotating}
\usepackage[numbers]{natbib}         
\usepackage[colorlinks]{hyperref} 
\hypersetup{colorlinks=true, linkcolor=black, citecolor = OliveGreen, urlcolor = black}
\makeatletter
\newcommand{\customlabel}[2]{%
   \protected@write \@auxout {}{\string \newlabel {#1}{{#2}{\thepage}{#2}{#1}{}} }%
   \hypertarget{#1}{#2}}
\makeatother
\usepackage{setspace}
\usepackage{mathrsfs}      

\usepackage{float}
\usepackage{setspace}
\usepackage{indentfirst}           
\usepackage[T1]{fontenc}

\allowdisplaybreaks[2]

\newcommand{\RR}{\mathbb{R}}      
\newcommand{\QQ}{\mathbb{Q}}
\newcommand{\CC}{\mathbb{C}}
\newcommand{\Qp}{\mathbb{Q}_p}

\newcommand{\Qbar}{\overline{\mathbb{Q}}}

\newcommand{\Fp}{\mathbb{F}_p}

\usepackage{listings}
\usepackage{comment}
\usepackage{manfnt}
\usepackage{longtable}

\makeatletter
\renewcommand{\@secnumfont}{\bfseries}
\makeatother

\usepackage[indentafter]{titlesec}
\titleformat{name=\section}{\centering\scshape\bfseries}{\thetitle.}{0.5em}{}
\titleformat{name=\subsection}[runin]{\bfseries}{\thetitle.}{0.5em}{}[.]
\titleformat{name=\subsubsection}[runin]{}{\thetitle.}{0.5em}{\itshape}[.]

\makeatletter
\def\@tocline#1#2#3#4#5#6#7{\relax
  \ifnum #1>\c@tocdepth 
  \else
    \par \addpenalty\@secpenalty\addvspace{#2}%
    \begingroup \hyphenpenalty\@M
    \@ifempty{#4}{%
      \@tempdima\csname r@tocindent\number#1\endcsname\relax
    }{%
      \@tempdima#4\relax
    }%
    \parindent\z@ \leftskip#3\relax \advance\leftskip\@tempdima\relax
    \rightskip\@pnumwidth plus4em \parfillskip-\@pnumwidth
    #5\leavevmode\hskip-\@tempdima
      \ifcase #1
       \or\or \hskip 1em \or \hskip 2em \else \hskip 3em \fi%
      #6\nobreak\relax
    \dotfill\hbox to\@pnumwidth{\@tocpagenum{#7}}\par
    \nobreak
    \endgroup
  \fi}
\makeatother

\makeatletter
\renewcommand{\@biblabel}[1]{[#1]\hfill}
\makeatother

\makeatletter
\def\@tocline#1#2#3#4#5#6#7{\relax
  \ifnum #1>\c@tocdepth 
  \else
    \par \addpenalty\@secpenalty\addvspace{#2}%
    \begingroup \hyphenpenalty\@M
    \@ifempty{#4}{%
      \@tempdima\csname r@tocindent\number#1\endcsname\relax
    }{%
      \@tempdima#4\relax
    }%
    \parindent\z@ \leftskip#3\relax \advance\leftskip\@tempdima\relax
    \rightskip\@pnumwidth plus4em \parfillskip-\@pnumwidth
    #5\leavevmode\hskip-\@tempdima
      \ifcase #1
       \or\or \hskip 1em \or \hskip 2em \else \hskip 3em \fi%
      #6\nobreak\relax
    \dotfill\hbox to\@pnumwidth{\@tocpagenum{#7}}\par
    \nobreak
    \endgroup
  \fi}
\makeatother

\numberwithin{equation}{subsection}

\numberwithin{equation}{subsection}

\newtheorem{theorem}[subsection]{Theorem}
\newtheorem{lemma}[subsection]{Lemma}
\newtheorem{coro}[subsection]{Corollary}
\newtheorem{conjecture}[subsection]{Conjecture}
\newtheorem{prop}[subsection]{Proposition}

\theoremstyle{definition}
\newtheorem{defn}[subsection]{Definition}
\newtheorem{exam}[subsection]{Example}

\newtheorem{assume}[subsection]{Assumption}

\newtheorem{remark}[subsection]{Remark}

\theoremstyle{remark}

\newcommand{\mZ}{\mathbb{Z}}

\newcommand{\mQ}{\mathbb{Q}}

\newcommand{\mC}{\mathbb{C}}

\newcommand{\mP}{\mathbb{P}}

\newcommand{\cS}{\mathcal{S}}

\newcommand{\fp}{\mathfrak{p}}

\newcommand{\fP}{\mathfrak{P}}

\newcommand{\brk}[1]{ \left\lbrace #1 \right\rbrace }

\newcommand{\tth}{^{\text{th}}}

\newcommand{\iso}{\cong}

\newcommand{\X}{\mathscr X}
\newcommand{\U}{\mathcal U}
\renewcommand{\O}{\mathcal O}
\newcommand{\ol}{\overline}
\newcommand{\FF}{\mathbb{F}}
\newcommand{\ZZ}{\mathbb Z}
\newcommand{\unr}{\text{unr}}
\newcommand{\nodes}{\text{ns}}

\newcommand{\ext}{\hookrightarrow}

\newcommand{\lrra}{\longrightarrow}

\def\quotient#1#2{\raise1ex\hbox{$#1$}{\Large/} \lower1ex\hbox{$#2$}}

\DeclareMathOperator{\sm}{sm}
\DeclareMathOperator{\ord}{ord}

\DeclareMathOperator{\codim}{codim}
\DeclareMathOperator{\Sp}{Sp}

\DeclareMathOperator{\rank}{rk}

\DeclareMathOperator{\tors}{tors}

\DeclareMathOperator{\Gal}{Gal}

\DeclareMathOperator{\red}{red}

\DeclareMathOperator{\Sym}{Sym}

\DeclareMathOperator{\vrt}{vert}

\DeclareMathOperator{\MV}{MV}
\DeclareMathOperator{\New}{New}

\DeclareMathOperator{\Trop}{Trop}

\DeclareMathOperator{\conv}{conv}

\DeclareMathOperator{\vol}{vol}

\DeclareMathOperator{\Lie}{Lie}
\DeclareMathOperator{\BC}{BC}
\DeclareMathOperator{\Ab}{Ab}

\DeclareMathOperator{\Per}{Per}
\DeclareMathOperator{\redu}{red}

\newcommand{\Cp}{\mathbb{C}_p}

\begin{document}

\title{Uniform bounds for the number of rational
points on symmetric squares of curves with low Mordell--Weil rank}

\author{Sameera Vemulapalli}
\address{Department of Mathematics, University of California, Berkeley, CA 94720-3840, USA}
\email{sameera.vemulapalli@berkeley.edu}

\author{Danielle Wang}
\address{Department of Mathematics, Massachusetts Institute of Technology, Cambridge, MA
02139-4307, USA}
\email{diwang@mit.edu}

\begin{abstract}
A central problem in Diophantine geometry is to uniformly bound 
the number of $K$-rational points on a smooth curve $X/K$ in terms
of $K$ and its genus $g$.
A recent paper by Stoll proved uniform bounds for the number of
$K$-rational points on a hyperelliptic curve $X$ provided that
the rank of the Jacobian of $X$ is at most $g - 3$.
Katz, Rabinoff and Zureick-Brown generalized his result to
arbitrary curves satisfying the same rank condition.

In this paper, we prove conditional uniform bounds on the number of rational points on the symmetric square of $X$ outside its algebraic special set, provided that the rank of the Jacobian is at most $g-4$. We also find rank-favorable uniform bounds (that is, bounds depending
on the rank of the Jacobian) in the hyperelliptic case. 
\end{abstract}
\date{\today}
\maketitle

\section{Introduction}
\label{Introduction}

Let $X$ be a curve of genus $g \ge 2$, defined over a number field $K$. Let $J$ be its Jacobian, and 
let $r$ denote the Mordell--Weil rank $\rank_{\ZZ} J(\QQ)$.
In 1983, Faltings famously proved the Mordell conjecture asserting the finiteness of the number of $K$-rational points of $X$.
This result leads to the natural question of the existence of uniform bounds. A central conjecture in Diophantine geometry is
\begin{conjecture}
\label{uniformityconjecture}
    There exists a constant $B(g, K)$ such that any smooth curve $X$ over $K$ of genus $g \geq 2$ has at most $B(g, K)$ rational points.
\end{conjecture}

Recent years witnessed progress on this uniformity conjecture for curves with low Mordell--Weil rank. Specifically, we have the following two results.
\begin{theorem}[\protect{\cite[Theorem~9.1]{stoll2013uniform}}]\label{thm:hypuniform}
Let $X$ be a hyperelliptic curve over $\QQ$ of genus $g\ge 3$ and suppose
    $r \le g - 3$. If $r = 0$, then
    \[
    	\#X(\QQ) \le 33(g-1) + 1.
    \]
If $r \geq 1$, then
\[
	\#X(\QQ) \le 8rg + 33(g-1) - 1.
\]
\end{theorem}

\begin{theorem}[\protect{\cite[Theorem~5.1]{katz2015uniform}}]\label{thm:anyuniform}
Let $X$ be any smooth curve over $\mQ$ of genus $g\geq 3$, and suppose that $r \leq g-3$. Then
$$\#X(\mathbb{Q}) \leq 84g^2 - 98g + 28.$$
\end{theorem}

Another question that Faltings' theorem naturally lends itself to is whether a similar finiteness
statement holds for higher degree points, i.e. for some $d \in \mathbb{Z}_{>0}$, points $P \in X(\overline{\mathbb{Q}})$ such that $[\mathbb{Q}(P) \colon \mathbb{Q}] \leq d$. 
Due to another deep result of Faltings \cite{faltings1994general} concerning points on subvarieties of abelian varieties, the answer is yes once we exclude points coming from the algebraic special set $\cS(\Sym^d X)$ (cf.~\autoref{special}) of the $d\tth$ symmetric power $\Sym^d X$ of the curve $X$.

The proofs of the above uniformity results involve variations of the $p$-adic method known as the Chabauty--Coleman method utilizing different theories of $p$-adic integration and Berkovich curves.
This $p$-adic technique has been generalized from curves to symmetric powers of curves through the work of Siksek \cite{siksek} and Park \cite{park}. 
Specifically, the work of Siksek established the general setup for symmetric power Chabauty--Coleman, and Park's work used tropical intersection theory to produce a conditional, effective bound on the number of points on the symmetric powers of curves with good reduction at $p$ lying outside of its algebraic special set.

\subsection{Statement of results}
In this paper, we prove conditional \emph{uniform} bounds on the number of points on the symmetric square of a curve $X$, with relatively low Mordell--Weil rank, lying outside of its algebraic special set. 

\begin{restatable}{thmm}{bound}
\label{finalboundcurves}
    Let $X$ be a smooth, projective, geometrically integral curve over $\QQ$ with genus $g \ge 4$ satisfying both $r \leq g-4$ and Assumptions
    \ref{zeroDimAssumption} and \ref{transverseAssumption}. Then the number of points in $(\Sym^2 X)(\QQ)$ lying outside
    its algebraic special set is at most
    \[
       288g^4 + \frac{1616}3g^3 - \frac{2900}9g^2 + \frac{11654}9g - \frac{4012}9.
    \]
\end{restatable}

We also obtain a rank-favorable uniform bound for hyperelliptic curves.

\begin{restatable}{thmm}{hyperelliptic}
    \label{hyperellipticrankfavorable}
    Let $H$ be a hyperelliptic curve over $\mQ$ of genus $g\geq 4$ satisfying both $r \leq g-4$ and Assumptions \ref{zeroDimAssumption} and \ref{transverseAssumption}. Then the number of points in $(\Sym^2 H)(\QQ)$ lying outside its algebraic special
    set is at most
    \[
96g^3r + 288g^3 + \frac{2192}{9}g^2r + \frac{69353}{36}g^2 - \frac{1184}{3}gr - \frac{14699}{9}g + \frac{736}{3}r + \frac{8053}{9}.
    \]
\end{restatable}

From this we obtain an immediate corollary concerning rational torsion packets of $\Sym^2 H$. Recall that a \textbf{rational torsion packet} for $\Sym^2 H$ is the inverse image of the group of rational torsion points of the Jacobian $J$ of $H$ under an Abel--Jacobi map $\Sym^2 H\ext J$.

\begin{restatable}{coro}{torsionpackets}
    \label{torsionpacketcorollary}
    Let $H$ be a hyperelliptic curve over $\mQ$ of genus $g\geq 4$ satisfying Assumptions \ref{zeroDimAssumption} and \ref{transverseAssumption}. Then, the size of a rational torsion packet of $\Sym^2 H$ lying outside 
    its algebraic special set is at most
    \[
        288g^3 + \frac{69353}{36}g^2 - \frac{14699}{9}g + \frac{8053}{9}.
    \]
\end{restatable}

\subsection{Overview of Proof}
Our proof proceeds in the spirit of \protect{\cite{stoll2013uniform}} and \protect{\cite{katz2015uniform}}. 
Specifically, we first cover $(\Sym^2 X)(\Qp)$ with residue polydisks and polyannuli. 
To generalize symmetric power Chabauty to annuli, we extend  Park's results concerning tropical intersection theory for power series to Laurent series, which is made possible due to the robustness of Rabinoff's \cite{rabinoff} theory of polyhedral subdomains. 
Finally, we perform a case-by-case analysis of the common zeros of the $p$-adic integrals coming from Chabauty's method using our above results.

\subsection{Outline of paper}
We begin in \autoref{background} with a discussion of Chabauty and Coleman's method and state some relevant facts concerning $p$-adic integration. 
In \autoref{generalization}, we introduce the setup for symmetric square Chabauty needed to obtain uniform bounds
on the number of points.
In \autoref{partition}, we partition $(\Sym^2 X)(\Qp)$ into residue disks and annuli making use of Stoll's results and a combinatorial argument. 
In \autoref{boundingzerodimcomponents}, we briefly review tropical analytic geometry and generalize Park's work on tropical intersection theory to the setting of Laurent series. 
In \autoref{calculationofnumberofzeros}, we perform the actual calculations of the number of common zeros of the $p$-adic integrals coming from Chabauty's method using the Newton polygons and mixed volumes.
We conclude in \autoref{uniformbounds} by summing up all of our cases to achieve uniform bounds. 

\subsection{Acknowledgements}
The authors would like to express our gratitude to Jackson Morrow for supervising this research. We would also like to thank Joseph Gunther and Jackson Morrow for the project idea. We are also grateful to Noam Kantor, Joseph Rabinoff, and David Zureick-Brown for the invaluable discussions, explanations, and ideas. We also appreciate the useful feedback and editing Eric Katz and Joseph Gunther have given us. Finally, this research was supported by the National Science Foundation (grant number DMS-1557690)
at the Emory University Mathematics REU.

\section{Background}
\label{background}

\subsubsection{Notation}
Let $X$ be a smooth, projective, geometrically integral curve
over $\QQ$. Let $J$ be the Jacobian of $X$. Let $X_L$ or $J_L$ represent the base change of the curve
or the Jacobian respectively to any field extension $L$ of $\QQ$. Let $g$ and $r$ denote the genus and Mordell-Weil
rank of $X$ respectively. Let
$X(L)$ denote the $L$-rational points of $X$.
Let $p$ be an odd prime, and let $\mathbb{Q}_p$ denote the $p$-adics, and let $\mathbb{C}_p$
denote the completion of the algebraic closure of $\mathbb{Q}_p$.
Let $H^{0}(J_L, \Omega^{1})$ and $H^{0}(X_L, \Omega^{1})$
denote the vector spaces of regular differentials on $J_L$ and $X_L$, respectively.

We use $\X$ to refer to a model
of a curve $X$ over $\O_{K_{\fP}}$ for some valued field $K_{\fP}$ 
with residue field $\kappa$.
We say that the model $\X$ is
\textbf{proper} if the Zariski closure of any point $P \in
X(K_{\fP})$ contains exactly one point 
$\ol P$ in $\X_s \coloneqq \X_{\kappa}$ \protect{\cite[\S~2.1]{katz2013chabauty}}.
We say it is \textbf{regular} if $\X$ is regular. In particular,
regularity implies that $\ol P$ is a smooth 
point of $\X_s$ for all $P$.

Properness gives us a reduction map
\[
    \red_p\colon X(K_{\fP}) \to \X_s(\kappa).
\]
We call the preimage of a point under the reduction map a \textbf{residue tube}. 
Stoll \cite{stoll2013uniform} showed that by contracting the components in chains of $\mathbb P^1$'s of a proper regular model,
we can obtain a proper model $\X$ where all points reduce
to either smooth points or ordinary double points (nodes).
Then each residue tube
is either analytically isomorphic to a $p$-adic disk or a
$p$-adic annulus considered as a rigid analytic space, by the following theorem.

\begin{theorem}[\protect{\cite[Proposition~2.3]{Bosch1985}}]
   	With notation as above, let $\X$ be a proper model
    of $X$.
    Let $\fp$ be a uniformizer of $K_{\fP}$, and
    let $P \in \X_s(\kappa)$ be a point on the special
    fiber $\X_s$. If $P$ is a smooth point,
    then the preimage of $P$ under the reduction map is analytically
    isomorphic to $\fp\O_{K_{\fP}}$, a $p$-adic disk. If $P$ is a node, its
    preimage is analytically isomorphic to a
    $p$-adic annulus.
\end{theorem}

\subsection{The Chabauty--Coleman method}

For a detailed exposition of Chabauty and Coleman's 
method, we refer the reader to \cite{mccallum2007method}.

In 1941, Chabauty \cite{chabauty} showed that the number of rational points on $X$ is finite when $r \le g - 1$. His main idea was
to consider $X(\QQ)$ in the more tractable spaces
\begin{center}
\begin{tikzcd}
X(\mathbb{Q}) \arrow[r,hook] \arrow[swap, d, hook] & X(\mathbb{Q}_p) \arrow[d, hook] \\%
J(\mathbb{Q}) \arrow[r, hook] & J(\mathbb{Q}_p),
\end{tikzcd}
\end{center}
for any prime $p$.
Let $\overline{J(\mathbb{Q})}$ represent the $p$-adic closure of $J(\mathbb{Q})$ inside $J(\mathbb{Q}_p)$. This topological space carries the structure of a $p$-adic Lie group of dimension $r'$, which sits inside the compact $p$-adic Lie group $J(\QQ_p) \iso (\mZ_p)^g \oplus J(\mQ_p)_{\tors}$.
Note that the bottom arrow factors as 
$J(\mathbb{Q}) \hookrightarrow \overline{J(\mathbb{Q})} \hookrightarrow J(\mathbb{Q}_p)$.

The idea behind Chabauty's method is to bound
$\# (\ol{J(\QQ)} \cap X(\QQ_p))$ instead of 
$\# X(\QQ)$. 
Under the condition $r\leq g-1$, \protect{\cite[Lemma~4.2]{mccallum2007method}} tells us that $r' \le r \le g-1$. 
Using this fact, Chabauty constructed locally analytic functions $f$ on $X(\QQ_p)$ satisfying $f(P) = f(Q)$
for $P, Q \in \ol{J(\QQ)}\cap X(\QQ_p)$ reducing
to the same point on $X_{\mathbb{F}_p}(\FF_p)$. Then he used the fact that locally
analytic functions, which are not locally constant, cannot achieve the same value
infinitely often to deduce that $\#(\ol{J(\QQ)}\cap X(\QQ_p))$ is finite.

Later, Coleman used 
Newton polygons to give an effective bound on the number of
rational points.

\begin{theorem}[\protect{\cite[Theorem~4]{coleman}}]
    If $p$ is a prime of good reduction such that $p > 2g$ and $r \leq g-1$, then \[ \#X(\mathbb{Q}) \leq \#X_{\mathbb{F}_p}(\mathbb{F}_p) + (2g-2). \]
\end{theorem}

There have been several refinements of Coleman's bound, most of which either seek to remove the assumption of $p$ being of good reduction or to gain some dependence on the rank of the Jacobian.  Specifically  \protect{\cite[Corollary~1.2]{lorenzinintucker}} and \protect{\cite[Corollary~6.7]{stoll2013uniform}} accomplished this by utilizing the theory of  proper regular models and using Clifford's theorem rather than Riemann-Roch as the geometric input.



In the past five years, there have been major leaps proving \emph{uniform} bounds for the number of rational points on curves which satisfy a Chabauty-like rank condition. 
Note that the previous results d	o not produce uniform bounds for the following reasons. First, the smallest prime of good reduction could be arbitrarily large, and second, a \textit{regular} proper model of $X$ can have arbitrarily long chains of $\mP^1$'s on its special fiber.
To ameliorate these issues, we forgo regularity of the model $\X$ of $X$, with the drawback that we must also analyze the integrals on $p$-adic annuli. 

The first uniformity result comes from Stoll \cite[Theorem~9.1]{stoll2013uniform}, who proves rank-favorable (depending on the rank) uniform bounds for hyperelliptic curves of Mordell--Weil rank $r\leq g-3$. The work of Katz, Rabinoff, and Zureick-Brown \cite[Theorem~5.1]{katz2015uniform} extends Stoll's uniform bound to arbitrary curves satisfying this rank condition using non-Archimedean potential theory
on Berkovich curves and the theory of linear systems and divisors
on metric graphs. However, the bound is not rank-favorable.
For our purposes, the main input is $p$-adic integration, which we recall below.

\subsection{$p$-adic integration}
\label{padicintegral}
In this section, we define the abelian and the Berkovich--Coleman integrals and discuss how they differ. We also state conditions under which the two integrals agree, and why such a fact is
necessary to prove uniform bounds.

\begin{defn}
    Since $J$ is an abelian variety over $\CC_p$, we can consider the
    abelian logarithm, which is defined to be the unique
    $\mathbb{C}_p$-Lie group homomorphism $\log \colon
    J(\mathbb{C}_p) \rightarrow \mathbb{C}_p^g
    \simeq \Lie(J)$ such that its
    derivative $d\log \colon \Lie(J) \rightarrow \Lie (\Lie
    (J)) = \Lie(J)$ is the identity map. Thus, $\log$ is also
    a local diffeomorphism. It is also a well-known fact that
    $\Lie(J)$ is the dual of $H^{0}(J_{\mathbb{C}_p},
    \Omega^{1})$.
\end{defn}

Thus, we have a bilinear map
 \[
    \Lie(J) \times H^{0}(J_{\mathbb{C}_p}, \Omega^{1}) \rightarrow \mathbb{C}_P,
\]
and we denote the evaluation pairing by $\langle \cdot , \cdot \rangle$.

\begin{defn}[Abelian integral]
    For $P, Q \in J(\mathbb{C}_p)$ and $\omega \in H^{0}(J_{\mathbb{C}_p}, \Omega^{1})$, define
    \[
        \prescript{\Ab}{}{\int_{0}^{P} \omega} \coloneqq \langle \log(P), \omega \rangle
    \]
    and
    \[
        \prescript{\Ab}{}{\int_{Q}^{P} \omega} \coloneqq \prescript{\Ab}{}{\int_{0}^{P} \omega} - \prescript{\Ab}{}{\int_{0}^{Q} \omega}.
    \]
\end{defn}


Recall from above that to obtain uniform bounds, we must work with a 
model of the curve at a prime $p$ of possibly bad reduction. The problem we run into is that some of the residue tubes
are annuli, and as the abelian integral cannot be expressed as a power series on these annuli, Coleman's argument fails to carry over.

To overcome this problem, we consider both the abelian integral and the Berkovich--Coleman integral. We only need the Berkovich--Coleman integral on residue tubes, and therefore only define it on these spaces.

\begin{defn}[Berkovich--Coleman integral]
    Let $\omega \in H^{0}(J_{\mathbb{C}_p}, \Omega^{1})$.
    Let $P, Q \in J(\mathbb{C}_p)$ be in the same residue tube.
    Suppose the residue tube has local parameter $t$. Then we can write
    \[
        \omega = \sum_{n=-\infty}^{\infty} a_n t^n \frac{dt}{t}.
    \]
    If $\omega$ has no $dt/t$ term, let $f$ denote the function 
    \[
       f(t) \coloneqq \sum_{n \neq 0} \frac{a_{n-1}}{n} t^n,
    \]
    and define
    \[
    	\prescript{\BC}{}{\int_{Q}^{P} \omega} \coloneqq f(t(P)) - f(t(Q)).
    \]
\end{defn}

The abelian integral is the integral that is used to power the classical Chabauty--Coleman method, but the Berkovich--Coleman integral has the advantage that it can be evaluated as a Laurent series on an annulus, whereas the abelian integral may not have such a simple representation. In order to reap the benefits of both integrals, we look for situations in which they are equal. 

The crucial theorem is that for a codimension $2$ space of differentials, these two integrals evaluate to the same number. 

\begin{theorem}
    [\protect{\cite[Proposition~7.3]{stoll2013uniform}}]
    \label{codimensionone}
    Let $A_{\ol P}$ be an annulus in $X(\mathbb{C}_p)$. Here $\ol P \in \X_s(\ol\Fp)$ is
the reduction of a point $P \in X(\Cp)$. Let $V$ be the subspace of $H^{0}(J_{\mathbb{C}_p}, \Omega^{1})$ 
    consisting of all $\omega$ such that 
    \[
        \prescript{\Ab}{}{\int_{P}^{Q} \omega} = \prescript{\BC}{}{\int_{P}^{Q} \omega}
    \] 
    for all $P, Q \in A_{\ol P}$, and such that the Laurent series
    expansion of $\omega$ on $A_{\ol P}$ has no $dt/t$ term (note that we view $A_{\ol P}$ as a subset of
    $J(\CC_p)$ via an Abel--Jacobi map).
    Then $V$ has codimension at most $2$. 
\end{theorem}

\section{Symmetric power Chabauty}
\label{generalization}
\subsubsection{Notation}
Define the symmetric power
$\Sym^{d}X$ to be the 
quotient of the $d$-fold Cartesian product of $X$ by the 
action of the symmetric group $S_d$. 
A point of $\Sym^d X(\QQ)$ can be represented as a $\Gal(\Qbar/\QQ)$-stable
multiset of $d$ elements of $X(\Qbar)$.

\subsection{Excluding the algebraic special set}
Siksek \cite{siksek} and Park \cite{park} generalized Chabauty's method to the
symmetric power $\Sym^d X$ of a curve $X$ satisfying $r\leq g -d$. 
An immediate problem of 
applying Chabauty's method to higher dimensional varieties is that the Albanese could be trivial. This problem is easily resolved when considering $\Sym^d X$ since its Albanese is the Jacobian of $X$. 
We also need the following assumption.

\begin{assume}
\label{dlessthang}
We need $0 < d < g$.
\end{assume}

This ensures that the image of $\Sym^d X$ under an
Abel--Jacobi map is a proper closed subvariety of the Jacobian of $X$.
Another immediate problem is that $\#\Sym^{d}X(\mathbb{Q})$ may 
be infinite. Park eliminates this problem by 
excluding the points in the algebraic special set 
of $\Sym^d X$.

\begin{defn}
\label{special}
Let $Y$ be a projective variety over $\overline{\mathbb{Q}}$. 
The \textbf{algebraic special set} of $Y$ is the Zariski closure of the union of 
the images of all nonconstant rational maps $f \colon G
\rightarrow Y$ of group varieties $G$ defined over
$\overline{\mathbb{Q}}$ into $Y$. We denote the special set of
$Y$ by $\mathcal{S}(Y)$. For a projective variety $Y$ over
$\mathbb{Q}$, let $\mathcal{S}(Y)$ be the closed subscheme of Y
whose base extension $\mathcal{S}(Y)_{\overline{\mathbb{Q}}}$ is
$\mathcal{S}(Y_{\overline{\mathbb{Q}}})$.
\end{defn}

By a result of Faltings \cite{faltings1994general}, we have that $\#(\Sym^d X \setminus \mathcal{S}(\Sym^d X))(\mathbb{Q})$ is finite, and we can bound it using a Chabauty-like technique.

\subsection{Generalization of Park's setup}
In the remainder of this section, we extend Park's work to  the case of arbitrary reduction modulo $p$. Let $K'$ be the compositum of all extensions of $\Qp$ of degree at most $d$. By the proof of \protect{\cite[Proposition~5.3]{stoll2013uniform}} (which we state in \autoref{partition}), there exists a proper model $\X/\O_{K'}$ of $X_{K'}$ whose residue tubes are the disks and annuli in said proposition. 
Then for any $d \ge 1$, this gives us a reduction map
\begin{align*}
	\red_p\colon (\Sym^d X)(\Qp)&\lrra (\Sym^d \X_s)(\FF_p) \\
    \{Q_1,\dots,Q_d\} &\mapsto \{\red_p'(Q_1),\dots,\red_p'(Q_d)\},
\end{align*}   
where $\red_p'$ denotes the reduction map $X(K') \to 
\X_s(\FF_{p^d})$ given by the model.
This is well-defined because $Q_1,\dots, Q_d \in X(K')$
by definition of $K'$.

We call the preimage of a point $\{\ol P_1,\dots,\ol P_d\}$
under this reduction map a \textbf{residue polytube}. 
If $\ol P_1,\dots,\ol P_d$ are all smooth, we call the
preimage a \textbf{polydisk}.

Consider any $\omega \in H^{0}(J_{\mathbb{Q}_p}, \Omega^{1})$. 
We define the function on $(\Sym^d X)(\Cp)$  
\begin{align*}
    \eta_\omega \colon (\Sym^d X)(\Cp) &\to \Cp \\
    \{Q_1,\dots,Q_d\} &\to \prescript{\Ab}{}{\int_0^{[Q_1 + \cdots + Q_d - O]} \omega},
\end{align*}
where $O$ is some fixed degree $d$ divisor of $X_{\QQ}$, and
the integral is as in \autoref{padicintegral}.
Siksek noted that if the residue polytube is a polydisk, we can
expand $\omega$ on the residue polytube in local coordinates as
\[
    \omega = \omega(u_1, \dots, u_d) = \sum_{i=1}^{d}\omega_i(u_1, \dots, u_d)du_i \in \mathbb{Z}_p[\![u_1, \dots, u_d]\!],
\]
and it can be shown that $\eta_\omega$ is given by the formal 
antiderivative of $\omega$ on residue polydisks.

The main obstruction we face when $p$ is of bad reduction is that
preimages under the reduction map may now be annuli.
We use the Berkovich--Coleman integral because it has a Laurent series representation on a residue annulus. However, we want $d$ integrals which actually equal the abelian integral,
so we work under the rank condition $r \leq g-d-2$ instead of Park's rank condition $r \leq g-d$.

\begin{coro}
\label{choosegoodomegas}
    Let $A_{\ol P}$ be an annulus coming from a node $\ol P$.
    If $r \leq g-d-2$, there exist $d$ linearly independent
    differentials $\omega_1, \dots, \omega_d$ on $J_{\Qp}$ which
    vanish on $\log \ol{J(\QQ)}$, such that for all $\omega$ in the span of $\omega_1,\dots ,\omega_d$
    \[
        \prescript{\Ab}{}{\int \omega} = \prescript{\BC}{}{\int \omega},
    \]
    and the corresponding Laurent series on $A_{\ol P}$
    have no $dt/t$ term.
\end{coro}
\begin{proof}
    Let $V$ be the subspace of $H^{0}(J_{\mathbb{Q}_p},
    \Omega^{1})$ consisting of all $\omega$ such that
    $\prescript{\Ab}{}{\int_{P}^{Q} \omega} =
    \prescript{\BC}{}{\int_{P}^{Q} \omega}$, and such that
    the Laurent series expansion of $\omega$ has no
    $dt/t$ term.
    By \autoref{codimensionone}, $V$ has codimension at most 2. 
    Note that one codimension comes from the condition of
    having no $dt/t$ term.
    As $\dim
    H^{0}(J_{\mathbb{Q}_p}, \Omega^{1}) = g$, there is a $g-2$
    dimensional subspace on which the integrals are equal. 
	Since $r \leq g-d-2$, we can 
    choose $d$ linearly independent differentials in this subspace
    that vanish on $\log \overline{J(\mathbb{Q})}$.
\end{proof}

\begin{defn}
    We say that a differential $\omega$ is
    \textbf{good} (on $A_{\ol P}$) if the corresponding
    functional on the Lie algebra vanishes on $\log(\overline{J(\mathbb{Q})})$ and the Laurent series
    expansion on $A_{\ol P}$ has no $dt/t$ term.
\end{defn}

In particular, as all the residue expansions of our $\omega_i$ do not
contain a $dt/t$ term, we note that the residue expansions of the
$\eta_i$ can be obtained using Berkovich--Coleman integration, which can be computed by formally anti-differentiating.


The rank condition allows one to find not one good differential but $d$ linearly independent good differentials $\omega_1,\dots,\omega_d$ from which, using
the abelian integral and composition with an Abel--Jacobi map, one can
construct $d$ locally analytic functions $\eta_i
\coloneqq \eta_{\omega_i}$, $1 \le i \le d$, on $(\Sym^d X)(\Cp)$ whose common zero set contains the $\mQ$-rational points of $\Sym^d X$. 

In order to bound the number of common zeros, 
we make the same assumption made by Park.
Let $\Lambda_X$ denote the $d$-dimensional vector space spanned by the $d$ linearly independent forms $\omega_1,\dots ,\omega_d$. Define 
\[
    (\Sym^{d}X)^{\Lambda_X} \coloneqq \brk{P \in (\Sym^d X)(\mC_p) : \eta_\omega(P) = 0, \, \forall\, \omega \in \Lambda_X}.
\] 


\begin{assume}[\protect{\cite[Assumption~1.0.7]{park}}]
\label{zeroDimAssumption}
The positive dimensional components of $(\Sym^{d}X)^{\Lambda_X}$ are contained in the $\CC_p$ points of the algebraic special set of $\Sym^{d}X$.
\end{assume}

\begin{lemma}[\protect{\cite[Corollary~4.2.4]{park}}]
Under Assumptions \ref{dlessthang} and \ref{zeroDimAssumption}, we have that the $\QQ$-rational points of $(\Sym^d X) \setminus \mathcal S(\Sym^d X)$ are contained in the zero-dimensional components of $(\Sym^d X)^{\Lambda_X}$.
\end{lemma}


If we suppose \autoref{zeroDimAssumption},
the rational points of $\Sym^2 X$ outside of its
special set correspond to zero-dimensional components of
the common vanishing locus of the $\eta_i$, so it suffices
to bound these.  

\begin{remark}
	By \cite[Lemma~4.1]{gunther2018irrational}, \autoref{zeroDimAssumption} always holds when $J$ is geometrically simple and $J(\QQ)$ has rank at most $1$.
\end{remark}

\subsection{Expressing integrals as pure power or Laurent series}
The first step is to rewrite the multivariate Laurent series without any
mixed monomial terms.

\begin{defn}
We define a \textbf{pure} power (resp. Laurent) series to be a power (resp. Laurent) series with no mixed monomial terms.
\end{defn}

Siksek showed that on residue polydisks, we can write the integrals $\eta_i$ as pure power series. 
The same argument shows that on any residue polytube, we express the integrals $\eta_i$ as pure Laurent series. 
Consider the residue polytube $U \subseteq (\Sym^d \X_s)(\Cp)$ over $\{P_1,\dots,P_d\} \in (\Sym^{d}X)(\ol{\mathbb{F}_p})$. 
Let $U_i$ be the individual residue tube over $P_i$, and
let $t_i$ be a uniformizer which induces an analytic isomorphism from $U_i$ to a disk or annulus. 
Let $\{Q_1,\dots,Q_d\}$, $\{Q'_1,\dots,Q'_d\} \in U$ such that
$Q_i, Q_i' \in U_i$. Then, by translation invariance of $\omega$,
we can write:

\begin{align*}
\eta(\{Q_1, \dots, Q_d\}) & = \int_{O}^{\{Q_1,\ldots,Q_d\}} \omega \\
\displaybreak[0]
 & = \int_{0}^{[Q_1+\dots+Q_d - O]} \omega\\
 \displaybreak[0]
 & = \int_{O}^{\{Q_1',\dots,Q_d'\}} \omega + \int_{Q'_1}^{Q_1} \omega + \dots + \int_{Q'_d}^{Q_d} \omega \\
 \displaybreak[0]
 & = C + \int_{t_{1}(Q'_1)}^{t_{1}(Q_1)} (t_1^{-1})^\ast\omega + \dots + \int_{t_{d}(Q'_d)}^{t_{d}(Q_d)} (t_d^{-1})^\ast \omega, \\
\end{align*}
where $C = \int_O^{\{Q_1',\dots,Q_d'\}} \omega$ is a constant.
We have shown that if we choose $\omega$ as in \autoref{choosegoodomegas} then 
$\int_{t_{i}(Q'_i)}^{t_{i}(Q_i)} (t_i^{-1})^\ast\omega = f_i(t_i)$ for some one variable power or Laurent series $f_i$. Then
\[
	\eta(t_1,\dots,t_d) = C +  \sum_{i=1}^{d}  f_i(t_i).
\]
Thus we have obtained a pure representation of our
Laurent or power series. Note that this pure series is a function on 
$U_1\times\cdots\times U_d \subseteq X^d(\CC_p)$ not $(\Sym^d X)(\CC_p)$.

\begin{remark}
Note that by \protect{\cite[Lemma~2.3]{siksek}},
if $U_i$ is a $p$-adic disk, we can choose $t_i$ such that the power series expansion of the integral on $U_i$ has no
constant term, i.e. it vanishes at $0$. We call this  
``centering the disk.'' 
\end{remark}



Combining this section with \autoref{choosegoodomegas} gives the following generalization of \protect{\cite[Proposition~3.2.2]{park}}. 
The second part relates the number of zeros of these functions to the number of rational points in the residue
polytube. 


\begin{prop}
\label{setup}  
    For each residue polytube
    $U \subseteq (\Sym^d X)(\Qp)$,
    let $K$ be the compositum of the fields of definition of all of the points in $U$. 
    Then 
    there exist good
    $\omega_1,\dots,\omega_d \in H^0(J_{\QQ_p}, \Omega^1)$,
    such that for each $1\le i\le d$ there exists a pure
    Laurent series or power series $\eta_i$ 
    with coefficients in $K$ 
    satisfying 
    \[
        \int_{Q_1}^{Q_2}\omega_i = \eta_i(t(Q_2)),
    \]
    for some $Q_1 \in U$ and all $Q_2 \in U$, where
    $t$ is an analytic isomorphism from $U$ to a product of  disks and annuli. If 
    $$\redu_p(Q_1) = \redu_p(Q_2) = \{P_1,\dots,P_d\} = 
    \bigcup_{i=1}^r \mathcal S_i$$
    where $\mathcal S_i$ consists
    of $s_i$ copies of the same point
    in $\X_s(\Fp^d)$, 
    then each point in $U \cap \Sym^d X(\QQ)$ corresponds to $N = \prod_{i=1}^r (s_i)!$ 
    common zeros of $\eta_1,\dots,\eta_d$.
\end{prop}

\section{Partitioning $(\Sym^2 X)(\QQ_p)$ into disks and annuli}
\label{partition}

In this section, we  compute how many  residue disks and annuli cover $(\Sym^2 X)(\QQ_p)$. 
First, we recall Stoll's results about partitioning into disks and annuli, stated below.
\begin{prop}[\protect{\cite[Proposition~5.3]{stoll2013uniform}}]
    \label{stoll53}
    Let $X$ be any smooth projective geometrically integral curve of genus $g$ over a $p$-adic field $k/\mathbb{Q}_p$, and let $q$ be the size of the residue field. Then there is a number $0 \leq t \leq g$ such that $X(k)$ can be written as a disjoint union of the set of $k$-points of at most $(5q+2)(g-1)-3q(t-1)$ open disks and at most $2(g-1)+(t-1)$ open annuli in $X$.
\end{prop}

\subsubsection{Notation}
Recall that $p$ is an odd prime.
Note that $\mathbb{Q}_p$ only has three quadratic extensions, one of which is unramified and two of which are ramified. Let $\QQ_{p^2}$ denote the unramified extension of $\Qp$, and let $K_1$ and $K_2$ denote the other two ramified extensions.
Let 
\begin{align*}
D_1 &= (5p+2)(g-1) - 3p(t-1), \\
D_2 &= (5p^2 + 2)(g-1) - 3p^2(t-1), \\
\alpha &= 2(g-1) + (t-1). 
\end{align*}
Finally, let $P^\sigma$ denote the Galois conjugate
of a quadratic point $P$.

\subsection{Disks and annuli in the symmetric square case}
Let $X$ be a curve with rank $r \le g - 4$.
For each point $\{P_1, P_2\} \in (\Sym^{2}\X_s)(\mathbb{F}_p)$,
we fix an ordering and take the preimage in 
$X \times X$, which is a product of disks or annuli.

For our purposes, we need to understand the possible reduction types in the $d=2$ case.
A point in $\Sym^2 X(\QQ_p)\setminus \mathcal{S}(\Sym^2 X(\QQ_p))$ could reduce to one of the following:
\begin{enumerate}
\item
\begin{enumerate}
    \item $P_1 = P_2^\sigma \in \X_s(\FF_{p^2})\setminus \X_s(\FF_p)$ and
    both are smooth points on the special fiber;
    \item $P_1 = P_2^\sigma \in \X_s(\FF_{p^2})\setminus \X_s(\FF_p)$ and
    both are nodes on the special fiber;
\end{enumerate}
\item
\begin{enumerate}
    \item $P_1 = P_2 \in \X_s(\FF_p)$ is a smooth point on the special fiber;
    \item $P_1 = P_2 \in \X_s(\FF_p)$ is a node on the special fiber;
\end{enumerate}
\item
\begin{enumerate}
    \item $P_1 \neq P_2$, $P_1, P_2 \in \X_s(\FF_p)$ and both are smooth points on the special fiber;
    \item $P_1 \neq P_2$, $P_1, P_2 \in \X_s(\FF_p)$ both are nodes on the special fiber;
    \item $P_1 \neq P_2$, $P_1, P_2 \in \X_s(\FF_p)$ one is a node and the other
    is a smooth point on the special fiber.
\end{enumerate}
\end{enumerate}


Using Stoll's bounds on the number of residue disks and annuli,
we get the following results. We elaborate on this table below.
\renewcommand{\arraystretch}{2}
\begin{longtable}{| l | l | l | }\hline
    Case & Number of Disks/Annuli & Type of preimage \\\hline\hline
    1(a) & $\frac12 D_2$
    & $p\O_{\QQ_{p^2}} \times p\O_{\QQ_{p^2}}$
    \\\hline
    1(b) & $\frac12 \alpha$ & $A(\QQ_{p^2}) \times A(\QQ_{p^2})$\\\hline
	2(a) & $D_1$ & $p_{K_i} \O_{K_i} \times p_{K_i}\O_{K_i}$ \\\hline
    2(b) & $\alpha$ & $A(K_i) \times A(K_i)$ 
    \\\hline
    3(a) & $\displaystyle{\binom{D_1}{2}}$ & 
    $p\ZZ_p \times p\ZZ_p$ 
    \\\hline\noalign{\penalty-5000}\hline
    3(b) & $\displaystyle{\binom{\alpha}{2}}$ &$ A(\Qp) \times A(\Qp)$ \\\hline
    3(c) & $D_1\alpha$
     & $p\ZZ_p \times A(\Qp)$ 
     \\\hline
\caption{Partition data for $\Sym^2(\Qp)$}
\end{longtable}

\noindent\textbf{Case 1.} In the first case, $q = p^2$. Since $P$ determines
$P^\sigma$, we are counting pairs of conjugate
residue disks / annuli. Because each point has a distinct conjugate that is not itself, and is in fact in a different residue tube, we simply multiply Stoll's
bound by $\frac{1}{2}$. 
\\ \\
\textbf{Case 2.}
In the second case, $q = p$ because points which reduce to
$\FF_p$ must lie in a ramified extension of $\Qp$.
Thus, from \emph{each of the two ramified extensions} of
$\QQ_p$, we get
    $D_1$ disk $\times$ disk, and
    $\alpha$ annulus $\times$ annulus.
However, note that in the second case, there is a 
two-to-one correspondence between common zeros
of $\eta_1$ and $\eta_2$ and points in the residue
polydisk/polyannulus.
\\ \\
\textbf{Case 3.}
In the third case, $q = p$, and disk
$\times$ annulus is possible. We get 
    $\binom{D_1}{2}$ disk $\times$ disk,
    $\binom{\alpha}{2}$ annulus $\times$ annulus, and 
    $D_1\alpha$ disk $\times$ annulus.
    
We obtain bounds on each case by bounding the number of zeros in each case and multiplying it by the number of disks and annuli for that case. We then sum up these bounds to achieve a global bound.

\section{Bounding zero-dimensional components}
\label{boundingzerodimcomponents}

In this section, we briefly recall background on tropical analytic geometry, in particular Rabinoff's \protect{\cite{rabinoff}} theory of polyhedral subdomains. We conclude by discussing relevant lemmas from \protect{\cite{stoll2013uniform}} and \protect{\cite{katz2015uniform}} concerning the Laurent series representations of the $p$-adic integrals.

\subsubsection{Notation}
In this section, let the $K$ denote the field $\mathbb{C}_p$. 
Let $B_K^1 = \Sp K\langle x_1 \rangle$, where $\Sp$ is the functor that takes a quotient of a Tate algebra to a rigid analytic space whose points consist of the maximal spectrum of the quotient of the Tate algebra.

\subsection{Tropical analytic geometry}
In \protect{\cite{rabinoff}}, Rabinoff expounded on the relationship between tropical and rigid analytic geometry. Roughly speaking, for a ``nice'' polyhedron $P \subseteq \mathbb{R}^n$, Rabinoff defined an affinoid open sub-domain $\U(P)$ of the analytification of an affine toric variety associated to combinatorial data attached to $P$.
\begin{defn}
	We refer the reader to \protect{\cite[Notation~2.2, Definition~2.3 parts (i) and (iv)]{rabinoff}}
    for the definition
    of integral $\Gamma$-affine polyhedron, and to \textit{loc.~cit.} Definition  3.12 for the cone of unbounded directions.
\end{defn}
Throughout this section, let $P$ be an integral $\Gamma$-affine polyhedron in $\RR^d$.
\begin{defn}
     Let $\sigma$ be
    the cone of unbounded directions of $P$, and let
    $S_P = \sigma^\vee \cap \ZZ^d$. Then define
    \[
        K\langle \U(P)\rangle \coloneqq
        \left\{
            \sum_{u \in S_P} a_ux^u \biggr\rvert a_u \in K,
            v(a_u) + \langle u, v \rangle \to \infty 
            \text{ for all } v \in P
        \right\}.
    \]
    (This means that for fixed $v \in P$ and
    $N$, there are only finitely many
    $u \in S_p$ such that $v(a_u) + \langle u, v\rangle < N$.)
    
    By \protect{\cite[Proposition~6.9]{rabinoff}}, $K\langle\U(P)\rangle$
    is a $K$-affinoid algebra. One can think of $K\langle\U(P)\rangle$ as the set of power series that converge on the points of $K^n$ whose valuation belongs to the polyhedron $P$.
\end{defn}

\begin{exam}
    For $P = \{(x_1,\dots,x_d) \in \RR^d \mid x_i \ge m_i
    \text{ for } 1 \le i \le d, m_i \in \QQ_{> 0}\}$,
    $K\langle \U(P)\rangle$ is isomorphic to the usual Tate algebra
    over $K$.
\end{exam}

\begin{exam}
    If $P = \prod_{i=1}^n [r_i, s_i]$, then $K\langle \U(P)\rangle$
    is the algebra of Laurent series which converge on the points
    with valuation in $P$. \footnote{There is an error in \protect{\cite[Example~6.8]{rabinoff}}. The indexing of the summation should
    be over $\ZZ^n$.}
\end{exam}

\begin{defn}
    Given nonzero $f \in K\langle \U(P)\rangle$, define
    \[
        H(f) \coloneqq \{(u, v(a_u)) \mid u \in S_P, a_u \neq 0\}.
    \]
    We call this the \textbf{height graph} of $f$ with respect to $0$.
    For $w \in P$, define
    \[
        m_w(f) \coloneqq \min \{ (w, 1) \cdot (u, v(a_u)) \mid
        (u, v(a_u)) \in H(f) \}.
    \]
    Let \footnote{This is different from Rabinoff's definition since we use different sign conventions on the exponents of the Laurent series.}
    \[
        \vrt_w(f) \coloneqq \{(u, v(a_u) \mid u \in S_P, (w,1) \cdot (u, v(a_u)) = m_w(f) \}. 
    \]
    Define
    \[
    	\vrt_P(f) \coloneqq \bigcup_{w \in P} \vrt_w(f).
    \]
\end{defn}

Geometrically, $\vrt_w(f)$ can be thought of as the lower faces of the regular subdivision given by the valuation.

\begin{defn}
    For any $f \in K\langle \U(P)\rangle$, define
    \[
        \Trop(f) = \ol {\{ w \in P \mid \# \vrt_w(f) > 1 \}}.
    \]  
\end{defn}

Note that the above is not the classical definition of $\Trop(f)$, however it is equivalent to the classical definition by fundamental theorem of tropical geometry, which is \protect{\cite[Theorem~3.2.5]{maclagan2015introduction}}.

\begin{defn} 
    For $f_1,\dots,f_n \in K\langle\U(P)\rangle$, define
    \[
        V(f_1,\dots,f_n) \coloneqq \Sp(K\langle\U(P)\rangle/(f_1,\dots,f_n)).
    \]
\end{defn}

\begin{defn}
    Let $f_1,\dots,f_d \in K\langle\U(P)\rangle$, $Y_i = V(f_i)$, and $Y = \bigcap_{i=1}^d Y_i$. 
    Assume that $Y$ is zero-dimensional. 
    Then the intersection multiplicity of $Y_1,\dots, Y_d$
    at $w$ is defined as
    \[
        i(w, Y_1,\dots,Y_d) 
        = \dim_{K}H^{0}(Y \cap U_{\{w\}}, \mathcal{O}_{Y \cap U_{\{w\}}})
    \]
    where we view $\{w\}$ as a zero-dimensional polytope. 
    In simpler terms, this intersection multiplicity at $w$ 
    is the number of common zeros of the $f_i$ that have the 
    same coordinate-wise valuation as $w$, counted with
    multiplicity. 
    
    Also define $\gamma_w(f_i) = \pi(\conv(\vrt_w(f_i)))$,
    where $\pi \colon \RR^d \times \mathbb{R} \to \mathbb{R}^d$ is the projection
    onto the first coordinate, and $\conv$ denotes the convex
    hull of a set of points in $\RR^{d+1}$. 
    Let $\gamma_i = \gamma_w(f_i)$ when
    it is not ambiguous.
\end{defn}

\begin{defn}[Mixed volume]
    Let $P_1,\dots,P_d$ be bounded polyhedra in $\RR^d$.
    Define the function
    \[
        f(\lambda_1,\dots,\lambda_d)
        \coloneqq
        \vol(\lambda_1P_1 + \cdots + \lambda_dP_d),
    \]
    where $+$ is the Minkowski sum. The
    \textbf{mixed volume}
    of the $P_i$, denoted
    $\MV(P_1,\dots,P_d)$ is defined as the coefficient of the
    $\lambda_1\cdots\lambda_d$ term of
    $f(\lambda_1,\dots,\lambda_d)$.
\end{defn}

\begin{theorem}[\protect{\cite[Theorem~11.7]{rabinoff}}]
    \label{countmult}
    Suppose $f_1,\dots, f_d \in K\langle\U(P)\rangle$ have
    finitely many common zeros, and let $w \in \bigcap_{i=1}^d \Trop(f_i)$
    be an isolated point in the interior of $P$. Let $\gamma_i$ be as above.
    Let $Y_i = V(f_i)$. Then
    \[
        i(w, Y_1,\dots, Y_d) = \MV(\gamma_1,\dots,\gamma_d).
    \]
\end{theorem}
It can be shown that $\vrt_P(f_i)$ is finite, so \autoref{countmult}
implies that the $\gamma_i$, and hence the intersection multiplicity information depend only on a finite number of terms of $f_i$.
Thus, we can approximate each series by a polynomial.
The following definition and theorem allow us to bound the
number of common zeros of a set of (Laurent) polynomials.

\begin{defn}[Newton polygon]
	Given a polynomial $f = \sum_{u \in S} a_ux^u$,
	$S \subseteq \ZZ^d$ finite, define
	\[ \New(f) = \conv(\{u : a_u \neq 0\}) \subseteq \mathbb{R}^n. \]
\end{defn}

\begin{theorem}[\protect{\cite{bernstein}}]
    \label{bernstein}
    Let $f_1,\dots,f_d \in K[x_1^{\pm1},\dots,x_d^{\pm1}]$ be
    Laurent polynomials with finitely many
    common zeros. Then the number of common zeros with
    multiplicity of $f_i$ in $(K^{\times})^{d}$ is at most
    \[ \MV(\New(f_1),\dots,\New(f_d)). \]
\end{theorem}

The following theorem generalizes \protect{\cite[Theorem~5.3.13]{park}} 
to Laurent series.

\begin{theorem}
    \label{approximate}
	Let $f_1, \dots, f_d \in K\langle\U(P)\rangle$ have finitely many common zeros. Further assume that $\bigcap_{i=1}^d \Trop(f_i)$ consists only of isolated points.
	Let $S_1,\dots,S_d \subseteq \mathbb{Z}^d$ be finite sets
	such that $S_i$ contains $\pi(\vrt_P(f_i))$ for each $i$.
    Define the \textbf{auxiliary polynomials}
	\[
		g_i = \sum_{u\in S_i} a_u x^u.
	\]
	Then if $g_1,\dots,g_d$ have finitely many
	common zeros,
	\begin{align*}
		\#\left(v^{-1}(P^{\circ}) \cap \bigcap_{i=1}^d V(f_i)\right)
		&\le \MV(\New(g_1),\dots,\New(g_d)).
	\end{align*}
    Here, $P^\circ$ is the interior of $P$ and
    $v^{-1}(P^\circ)$ denotes the set of points in
    $(K^\times)^d$ with valuation in $P^\circ$.
\end{theorem}
\begin{proof}
    By construction $\gamma_w(f_i) = \gamma_w(g_i)$ for all 
    $i$ and $w \in P^\circ$.
    By \autoref{countmult}, the number of common zeros
    of the $f_i$ with valuation $w \in P^\circ$ is determined only by
    $\gamma_w(f_i)$. Thus, the number of common zeros of the $f_i$
    with valuation in $P^\circ$
    is equal to the number of common zeros of the $g_i$, with
    valuation in $P^\circ$. This is obviously less than or equal to the number of common zeros with any valuation.
    But by \autoref{bernstein} the \emph{total} number of common
    zeros of the $g_i$ in $(K^\times)^d$ is 
    \[ 
        \MV(\New(g_1),\dots,\New(g_d)). 
    \]
    This proves the theorem. Note that $\MV(\New(g_1),\dots,\New(g_d))$ may give us zeros with valuations outside of $P$, and thus our theorem only states an inequality.
\end{proof}	

\subsection{Deformation of Laurent series}
In order to count points outside the algebraic special set,
we need to count zero-dimensional components
of the vanishing locus of the $\eta_i$, even if there are
(infinite) positive-dimensional components.
In this subsection, we show that we can deform a set of
Laurent series to have finite intersection, without changing
their tropicalizations and $\gamma_i$.

\begin{defn}
    Let $P \subseteq \mathbb{R}^d$ be an integral $\Gamma$-affine polyhedron. For $f_1,\dots,f_d
    \in K\langle\U(P)\rangle$, define
    $N_{0}(f_1,\dots,f_d)$ to be the number of zero-dimensional
    components of $\bigcap_{i=1}^{d}V(f_i)$, counted with multiplicity.
    If $Y = V(f_1,\dots,f_d)$ is finite, then define
    \[ N(f_1,\dots,f_d) = \dim H^{0}(Y, \mathcal{O}_{Y}). \] 
\end{defn}


Note that for the rest of this paper, we say
that a series is non-degenerate if for every variable $t$, some power of $t$ appears with nonzero coefficient.

\begin{theorem}[\protect{\cite[Theorem~6.1.7]{park}}]
    \label{deformation}
	Let $f_1, \dots, f_d \in K\langle U(P)\rangle$ be non-degenerate
	Laurent series. There exist $g_1, \dots, g_d$ with
	$\Trop(f_i) = \Trop(g_i)$ and $\gamma_w(f_i) = \gamma_w(g_i)$
	for all $w \in P$ such that the $g_i$ have finitely many
	common zeros, and
	\[
		N_0(f_1, \dots, f_d) \le N(g_1, \dots, g_d).
	\]
	Furthermore, if the $f_i$ are Laurent polynomials,
	then the $g_i$ can be chosen to be Laurent polynomials
	as well, with $\New(g_i) = \New(f_i)$.
\end{theorem}
\begin{proof}
    The proof of \protect{\cite[Theorem~6.1.7]{park}} 
    does not use the condition that the expansion is a power series and not a Laurent series. The conditions $\Trop(f_i) = \Trop(g_i)$ and $\gamma_w(f_i) = \gamma_w(g_i)$ immediately imply that $\New(g_i) = \New(f_i)$.
    Thus, the proof immediately generalizes to the Laurent series case, but we will explain the main idea.
    
    We deform the $f_i$ one at a time, by inductively finding
    $g_1,\dots,g_r$ satisfying
    \begin{enumerate}
      \item $\Trop(f_i) = \Trop(g_i)$ and $\gamma_w(f_i) = \gamma_w(g_i)$ for $i \in \{1,\dots,r\}$
      \item $\codim \bigcap_{i=1}^{r}V(g_i) \geq r$
      \item $N_{0}(f_1,\dots,f_d) \leq N_{0}(g_1,\dots,g_r,f_{r+1},\dots,f_d)$
    \end{enumerate}
    for $1 \le r \le d$.
    The statement is clear for $r = 1$, by taking $g_1 = f_1$.
    Given $g_1,\dots,g_r$, let $C_1,\dots,C_\ell$ be
    the codimension $r$ irreducible components of $\bigcap_{i=1}^r V(g_i)$, and let $P_i \in C_i$ for
    $i = 1,\dots,\ell$, such that $P_i \neq 0$ in $K^d$. 
    By \protect{\cite[Lemma~5.6]{park}} there exists
    a polynomial $h$ such that $h$ does not vanish at
    any of the $P_i$, and $M(h) \subseteq M(f_{i+1})$
   (see \protect{\cite[Definition~5.4]{park}}), so that
   $g_{r+1}\coloneqq f_{r+1} + \epsilon h$ has the same tropicalization
   and $\gamma_w$ as $f_{r+1}$ for small enough $\epsilon$.

    We will also elaborate slightly on why these 
    deformations do not decrease the number of 
    zero-dimensional components.
    Let $G_i(t_1,\dots,t_d,\epsilon) = g_i(t_1,\dots,t_d)$ 
    for $1\le i \le r$, let 
    $G_{r+1} = f_{r+1}(t_1,\dots,t_d) + \epsilon h(t_1,\dots,t_d)$,
    and let $F_{i} = f_i(t_1,\dots,t_d)$ for $i > r + 1$.
   	Define $I$ to be the product of all non-maximal minimal prime ideals containing the ideal
    $(g_1,\dots,g_r,f_{r+1},\dots,f_d)$, and let $f$ be an element of $I$ which does not vanish on any of the zero-dimensional components, which exists by the prime avoidance theorem.
    Consider the projection onto the second factor
    \[
    	\alpha \colon \Sp K\langle\U(P)\rangle|_{|f| \ge A} \times B_K^1 \to B_K^1,	
    \]
    where $A = \min_{1 \le i \le \ell} |f(P_i)|$, and
    $\Sp K\langle \U(P)\rangle|_{|f| > A}$ denotes
    restriction to the set $\{|f| > A\}$.
    Let $Y = V(G_1,\dots,G_r,G_{r+1},F_{r+2},\dots,F_d)$.
    Then $Y \cap \alpha^{-1}(0)$ consists of all points $\mathfrak p$ which do not
    vanish at $f$ but vanish at $g_1,\dots,g_r$ and $f_{r+1},\dots,f_d$. 
    This implies that $\mathfrak p$ is not on a positive dimensional component, 
    because if it were then $\mathfrak p$ vanishes at $I \ni f$.
    
    Thus the size of $Y \cap \alpha^{-1}(0)$ is equal to the number
    of zero-dimensional components of 
    $V(g_1,\dots,g_r,f_{r+1},\dots,f_d)$. Then \protect{\cite[Theorem~10.2]{rabinoff}}, implies
    that for small $|\epsilon|$, this is also equal to the number of 	common zeros of $g_1,\dots,g_r$, $g_{r+1} = f_{r+1} + \epsilon h$, $f_{r+1},\dots,f_d$, away from the positive-dimensional components
    of $V(g_1,\dots,g_r,f_{r+1},\dots,f_d)$.
\end{proof}

\begin{remark}
    In \autoref{approximate}, if the auxiliary
    polynomials $g_i$ have infinitely many common
    zeros, we can choose Laurent polynomials
    $g_1', \dots, g_d'$ with finitely many common
    zeros by \autoref{deformation}. 
\end{remark}

\begin{remark}
    In order to bound $N(g_1,\dots,g_d)$, we have to
    also count ``degenerate'' zeros, where some of the
    coordinates are zero (cf.~\protect{\cite[Section~6.2]{park}}).
    Note that coordinates can only be zero on disks, not on annuli, so we only have to deal 
    with this case when one or more residue tubes are disks.
\end{remark}

\subsection{Relevant results concerning Laurent series expansions}

\begin{remark}
    Since we want to bound the number of common zeros of the $\eta_i$
    with valuations in $[1/e, a/e) \times [1/e, b/e)$ 
    where $a$ or $b$ could be infinite,
    it suffices to bound the number of common zeros of the $\eta_i$
    with valuations in $P = [1/e, M] \times [1/e, N]$ as $N \to a/e$, $M \to b/e$,
    and $M, N < \infty$.
    Note that then $\eta_i \in K\langle \U(P)\rangle$, and the 
    theorems in this section apply. In particular, the local expansions coming from symmetric power Chabauty lie in some affinoid algebra $K\langle \U(P)\rangle$.
\end{remark}

For the rest of this section, $K$ is a finite extension
of $\Qp$.
The following is a slight generalization of a definition of \cite{stoll2006}.
\begin{defn}
    [\protect{\cite[Section~6]{stoll2006}}]
    
    For $r > 0$, define
    \[
        \delta(r,k) = \max\{N \ge 0 \mid rv(k+1) + N \le rv(k + N + 1)\}.
    \]
\end{defn}

The proof of \cite[Lemma~6.1]{stoll2006} does not
depend on the fact that $e$ is the ramification
index; it can be any positive real number. 
Thus we have the following generalization.

\begin{lemma}[\protect{\cite[Lemma~ 6.1]{stoll2006}}]
    \label{deltabound}
    If $p > r + 1$, then $\delta(r, k) \le r\lfloor k/(p-r-1)\rfloor$.
    In particular
    \[
        k + \delta(r, k) \le \mu_r k
    \]
    where
    \[ \mu_r = \frac{p-1}{p-r-1}. \]
\end{lemma}

\begin{lemma}[\protect{\cite[Lemma~6.2.4]{park}}]
    \label{padicrolle}
    Let $K/\mathbb{Q}_p$ be an extension with ramification index $e$.
    Let $f \in K[\![ t ]\!]$ be such that $f' \in \O_K[\![ t ]\!]$
    and such that $f$ has no constant term. 
    Let $k - 1 = \ord_{t=0} (f' \bmod{\pi_K})$, where
    $\pi_K$ is a prime of $\O_K$. Let $r < 1/e$.
    Then if $w \in (r, \infty)$ and $u > k + \delta(1/r, k-1)$ or $u=0$,
    then $(u, v(a_u)) \notin \vrt_w(f)$.
\end{lemma}
\begin{proof}
    By definition $u > k + \delta(1/r, k-1)$ means 
    \[
        \frac{1}{r}v(k) + u-k > \frac{1}{r}v(u),
    \]
    so 
    \[
        v(u) - v(k) < r (u-k) < w (u-k),
    \]
    or
    \[
        v(a_u) + uw > v(a_k) + kw,
    \]
    since $v(ua_u) \ge 0$, and $v(ka_k) = 0$.
    This shows that $(u,v(a_u)) \notin \vrt_w(f)$. 
    Since $a_0 = 0$, indubitably $(0, v(a_0)) \notin \vrt_w(f)$.
\end{proof}
\begin{remark}
\label{rmk:delta}
Note that
\[
	\lim_{1/r \to e} k + \delta(1/r, k-1) \le \lim_{1/r \to e} \mu_rk = \mu_e k.
\]
\end{remark}

Thus any power series can be approximated by a polynomial whose degree is less than $k + \delta(e, k-1)$. For Laurent series, we have the following analogous results from \cite{katz2015uniform}.

\begin{defn}[\protect{\cite[Definition~4.3]{katz2015uniform}}]
	For $r > 0$, $N_0 \in \ZZ$, and $p$ a prime,
    define $N_p(r, N_0)$ to be the smallest positive integer
    $N$ such that for all $n \ge N$, we have
    \[
    	r(n - N_0) > \lfloor \log_p(n) \rfloor.
    \]
\end{defn}

\begin{remark}
\label{rmk:Npr}
	By \cite[Remark~4.5]{katz2015uniform}, we have
    \[
    	\lim_{s \nearrow r} N_p(s, N_0) = N_p(r, N_0).
    \]
\end{remark}

\begin{lemma}
	If $r = 1/2$, $p \ge 5$ and $N_0 \ge 1$, then
    $N_p(r, N_0) \le 2N_0$.
\end{lemma}
\begin{proof}
	This can be easily checked from the inequalities preceding inequalities (4.6) in \cite{katz2015uniform}.
\end{proof}

\begin{theorem}[\protect{\cite[Corollary~4.18]{katz2015uniform}}]
\label{krzb418}
	Let 
    \[
    	A_{\ol{P}} = \{z \in K \mid 0 < v(z) < b/e \}
    \]
    be an annulus coming from a node, where
    $e$ is the ramification index of the field of definition $K$ of $A_{\ol p}$. Let
	$\omega$ be a good differential $1$-form contained in
    $V$, where $V$ is as in \autoref{codimensionone}. Then for $r < 1/e$ the
	number of zeros of $\eta_\omega$ on 
    \[
    	(A_{\ol P})_r \coloneqq \{ z \in K \mid r < v(z) < b/e - r \}
    \] 
    is at most $2N_p(r, 2g-2)$.
\end{theorem}

\begin{remark}
	This is equivalent to the statement that there
    there exists an interval $[c_1, c_2]$ of length at most
    $2N_p(r, 2g-2)$ such that $u \notin [c_1, c_2]$ implies 
    $u \notin \vrt_w(\eta_\omega)$ for $w \in (r, b/e - r)$, because \autoref{countmult} counts with multiplicity.
\end{remark}

\begin{lemma}
    \label{pureseries}
    Let $P$ be a polyhedron, and
    let $F = f_1(t_1) + \dots + f_d(t_d)$ be a pure
    Laurent series in $K\langle \U(P)\rangle$.
    Furthermore assume that $f_1(t_1)$ is the only series
    with a constant term (or possibly not). Let
    $w  \in P$. Then
    \[
        \vrt_w(F) \subseteq \vrt_{w}(f_1)
        \cup\cdots\cup \vrt_{w}(f_d).
    \]
\end{lemma}
\begin{proof}
	Note that we are considering the $f_i$ as a function of $d$ variables implicitly and thus $\vrt_w(F)$ lies in the same ambient space as $\vrt_w(f_i)$.
    Using the same notation as in \autoref{setup},
    let $u \in \vrt_w(F)$. Since $F$ is a pure Laurent series,
    we can write $u = (0,\dots,u_i, \dots, 0)$. 
    If $u_i = 0$, then $u \in \vrt_w(f_1)$. Otherwise, $u \in \vrt_w(f_i)$. 
\end{proof}

Because of the above two lemmas, we can bound our Newton polygons by the convex hull of all the vertices of all the pure Laurent series.

\begin{remark}
    In the situation of \autoref{setup},
    if $U = U_1\times\cdots\times U_d$, and $U_j$
    is a disk, then $\eta_i = f_{i,1}(t_1) + \cdots + f_{i,d}(t_d)$
    can be chosen so that $f_{i,j} \in K[\![t_j]\!]$ has
    no constant term and we can take $f = f_{i,j}$ in \autoref{padicrolle}.
 By \cite[Corollary~4.18]{katz2015uniform}, we can actually add a 
    constant to $\eta_\omega$ and the statement in \autoref{krzb418} still holds (since it only depends on $d\eta_\omega$).
     Thus, if $U_j$ is an annulus, then we can take $\eta_\omega + c_{i,j} = f_{i,j} + c_{i,j}$ in \autoref{krzb418}, where we choose constants $c_{i,j}$ so that
     their sum is zero, and at most one $f_{i,j} + c_{i,j}$
     has nonzero constant term.
\end{remark}

\section{Calculation of the number of zero-dimensional components}
\label{calculationofnumberofzeros}
In this section we bound the number of zero-dimensional
components of the vanishing locus of the
$\eta_i$ on each tube, the same way as Park did, 
by taking the mixed volume of the convex hulls of
the vertices of the individual components of the pure Laurent
series or power series. 
Throughout this section, assume $p\ge 5$ is a prime. 

\subsubsection{Notation}
For a fixed residue polytube $U$, consider the pullback $U_1\times U_2$, given by fixing an order and pulling back from the symmetric square to the Cartesian product of $X$. 
Let $\omega_1$ and $\omega_2$ be the differentials given
by \autoref{setup}, and let the expansions of the 
corresponding integrals as pure Laurent or power series on
$U_1\times U_2$ be
\[\eta_1 = f_{1,1}(t_1) + f_{1,2}(t_2), \quad \eta_2 = f_{2,1}(t_1) + f_{2,2}(t_2),\] 
where the first index corresponds to the differential and the second index corresponds to the disk/annulus.

We wish to apply the results of Section 5, in particular Theorems \ref{approximate} and \ref{deformation}, to
determine the common zeros of $\eta_1$ and $\eta_2$. In \cite{park}, there are technical difficulties
which arise because of a misapplication of \cite[Theorem 11.7]{rabinoff}. We do not wish
to address this in our work. Rather we include an additional assumption which
allows us to ignore these difficulties.

\begin{assume}
\label{transverseAssumption}
There exist linearly independent, good differentials $\omega_1, \dots, \omega_d$ and a cover of $(\Sym^d X)^{\Lambda_X}$ such that locally on each affinoid, the power series (or Laurent series) $\eta_1, \dots, \eta_d$ coming from integrating have tropicalizations $\Trop(\eta_1), \dots, \Trop(\eta_d)$ where the intersection $\bigcap_{i=1}^d \Trop(\eta_i)$ is transverse, i.e., consists only of isolated points.
\end{assume}

\begin{remark}
Furthermore, \cite[Theorem~11.7]{rabinoff} requires 
that $w$ is in the interior of $P$. To deal with 
this, we note that $[1/e, (a-1)/e] \times [1/e, (b-1)/e]$
is contained in the interior of of
$[r, a/e - r] \times [r, b/e - r]$ for
$r < 1/e$. Then we can appply Theorems \ref{approximate} and \ref{deformation}, and
use Remarks \ref{rmk:delta} and \ref{rmk:Npr}
to take the limits of the bounds.
\end{remark}

If $U_j$ is a disk, let $$k_{i,j} - 1 = \ord_{t_j=0} (f'_{i,j} \bmod{p}).$$
Let $D_1$, $D_2$, and $\alpha$ be as in \autoref{partition}. Let $\mu_e$ be as in \autoref{deltabound}. 

If $K$ is a finite extension of $\QQ_p$,
let $\sm(K)$ and $\nodes(K)$ 
denote the smooth points and nodes respectively in the
image of the reduction map
\[ X(K) \to \X_s(\FF_q), \] where $\FF_q$ is the residue field
of $K$.

Before we proceed with our computations, we recall some useful facts.

\begin{defn}
    For an $n \times n$ matrix $A = (a_{i,j})$, define
    the \textbf{permanent} of $A$ to be
    \[ \Per(A) = \sum_{\sigma \in S_n}\prod_{i=1}^n a_{i,\sigma(i)}. \]
\end{defn}

\begin{lemma}[\protect{\cite[Lemma~A.4]{mccallum2007method}}]
\label{diskriemannroch}
Let $V$ be as in \autoref{codimensionone}. Then
\[
	\sum_{P \in \sm(K)} \ord_P\omega \le 2g-2
\]
for all nonzero $\omega \in V$.
\end{lemma}

We split into $7$ cases as described in \autoref{partition}.
Recall that we
have three types of points in
$(\Sym^{2}\X_s)(\overline{\mathbb{F}_p})$ that quadratic points 
reduce to. For a point $\{Q_1, Q_2\}$ reducing to $\{P_1,P_2\}$, let $e$ be the ramification index of the field of definition of $Q_1$ and $Q_2$. 
We can check that given the type of reduction, the ramification index is uniquely
determined. 
In particular, $e = 1$, $2$, $1$ in Cases 1, 2, and 3, respectively.
\subsection{Case 1(a)}
In this case $e = 1$.
Consider a residue polydisk $U_1\times U_2$ which is the preimage of $\{P, P^\sigma\}$ for some smooth points $P, P^\sigma \in \X_s(\mathbb{F}_{p^2})\setminus \X_s(\mathbb{F}_p)$. By \autoref{padicrolle} and \autoref{pureseries},
$\eta_1$ and $\eta_2$ have auxiliary polynomials 
with Newton polygons whose limit is contained in
\[
    X_i = \conv(e_1, e_2, a_{i,1}e_1, a_{i,2}e_2),
\]
where $a_{i,j} = k_{i,j} + \delta(e,k_{i,j}-1)$,
and $e_i$ is the $i$-th standard vector.
(We can take $X_i$ to be these quadrilaterals instead
of the triangles in \cite{park}, because $\eta_1$ and $\eta_2$ can be chosen, by centering the disks, so that they have zero constant terms.)

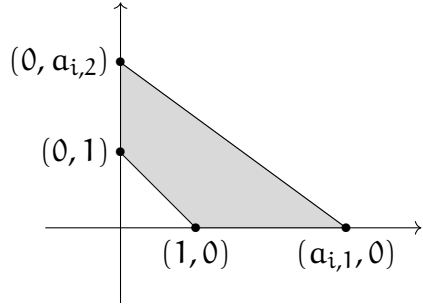
\begin{figure}[h]
\begin{tikzpicture}[scale=1]
    \draw[black, fill=gray, opacity=0.3] (1,0)--(3,0)--(0,2.2)--(0,1);
    \draw[->] (-1,0) -- (4,0);
    \draw[->] (0,-1) -- (0,3);
    \node at (0,1) {\tiny\textbullet};
    \node [left] at (0,1) {$(0,1)$};
    \node at (1,0) {\tiny\textbullet};
    \node [below] at (1,0) {$(1,0)$};
    \node at (0,2.2) {\tiny\textbullet};
    \node [left] at (0,2.2) {$(0,a_{i,2})$};
    \node at (3,0) {\tiny\textbullet};
    \node [below] at (3,0) {$(a_{i,1}, 0)$};
    \draw (1,0)--(0,1);
    \draw (3,0)--(0,2.2);
\end{tikzpicture}
\caption{The Newton polygon is contained in $X_i$.}
\end{figure}

Note that in this case, if
$\{Q_1, Q_2\}$ reduces to $\{P, P^\sigma\}$
with $P \in \X_s(\FF_{p^2}) \setminus \X_s(\FF_p)$,
then we must have $Q_1 = Q_2^\sigma$ as $\{Q_1, Q_2\}$ is a rational point on the symmetric square, so they are either
both zero on both nonzero --- hence the common
zeros of the $\eta_i$ which correspond to points in the residue
polydisk are either in $(K^\times)^2$ or equal to
$(0,0)$. 

\begin{figure}[h]
\begin{tikzpicture}[scale=1]
    \draw[black, fill=gray, opacity=0.3] (1,0)--(3,0)--(0,2.2)--(0,1);
    \draw[->] (-1,0) -- (4,0);
    \draw[->] (0,-1) -- (0,3);
    \node at (0,1) {\tiny\textbullet};
    \node [left] at (0,1) {$(0,\lambda_1 + \lambda_2)$};
    \node at (1,0) {\tiny\textbullet};
    \node [below] at (1,0) {$(\lambda_1 + \lambda_2,0)$};
    \node at (0,2.2) {\tiny\textbullet};
    \node [left] at (0,2.2) {$(0, \lambda_1 a_{1,1} + \lambda_2 a_{2,1})$};
    \node at (3,0) {\tiny\textbullet};
    \node [below] at (4,0) {$(\lambda_1 a_{1,2} + \lambda_2 a_{2,2}, 0)$};
    \draw (1,0)--(0,1);
    \draw (3,0)--(0,2.2);
\end{tikzpicture}
\caption{Minkowski sum of the $\lambda_i X_i$.}
\end{figure}
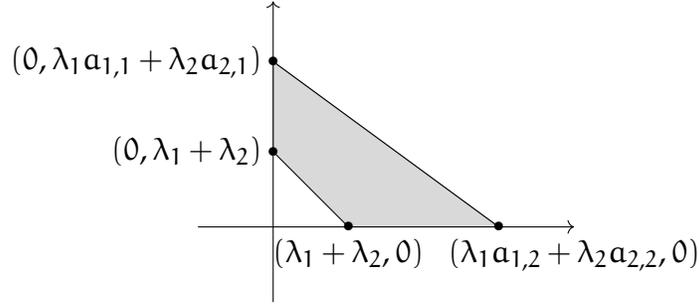
Thus, $N_0(\eta_1,\eta_2)$
is at most the mixed volume of the $X_i$, plus one 
for the possible solution at $(0,0)$, which we can directly compute to be
\[
    N_0(\eta_1,\eta_2) \le \frac{1}{2}(a_{1,2} a_{2,1} + a_{2,2} a_{1,1} - 2) + 1
    = 
    \frac12(\Per(a_{i,j}) - 2) + 1
\]
with $a_{i,j}\le \mu_1 k_{i,j}$ as above. 
By definition,
\[
    k_{i,j} = \ord_P\omega_i + 1,
\]  
and
\[
    a_{1,1}a_{2,2} + a_{1,2}a_{2,1} \\
    \leq \mu_{1}^{2}(k_{1,1}k_{2,2} + k_{1,2}k_{2,1}).
\]
So if we sum up across all the disks involved in Stoll's bound, noting
that the choice of one disk determines the other since they must be centered at conjugate points, so only one summation is needed, we get
\begin{align*}
\sum_{U_1\times U_2} N_0(\eta_1,\eta_2) &\le \frac12 \sum_{P \in \sm(\QQ_{p^2})}\frac{1}{2}(\mu_{1}^{2}(k_{1,1}k_{2,2} + k_{1,2}k_{2,1}) - 2) + 1 \\
    &\le \frac{\mu_1^2}{2}\sum_{P \in \sm(\QQ_{p^2})} 
    (\ord_P\omega_1 + 1)(\ord_{P^\sigma}\omega_2 + 1)\\
     &\le \frac{\mu_1^2}{2}(D_2 + (2g-1)^2),
\end{align*}
where the inequality comes from \autoref{diskriemannroch}.
By the second half of \autoref{setup}, $N_0(\eta_1,\eta_2)$
is an upper bound for the number of rational points on the symmetric square that pulls back to $U$ in this case, so the total number of rational points which
reduce to a pair of conjugate smooth points is at most

\newsavebox\onea

\savebox\onea{$\displaystyle{
\frac{\mu_1^2}{2}(D_2 + (2g-1)^2)
}$}

\begin{equation}
    \label{case1abound}
    \usebox\onea .
\end{equation}


\subsection{Case 2(a)}
In this case $e = 2$.
Consider residue tube $U$ which is the preimage of $\{P, P\}$ for some smooth point $P \in \X_s(\mathbb{F}_p)$.
We fix one of the ramified extensions $K_i$ and pull back
to the Cartesian product to
get a polydisk $U_1\times U_2 \subseteq X^2(K_i)$.

In this case, we must count degenerate zeros. 
We bound the number of degenerate zeros in the following way. 
Since our power series are of the form
$f_{i,1}(t_1) + f_{i,2}(t_2)$, where the
$f_{i,j}$ all have no constant term, 
if without loss of generality we set $t_1 = 0$, we are now looking for 
the number of zeros of the $f_{i,2}$, which we crudely
bound by the sum of the number of common zeros
of each $f_{1,2}$ and $f_{2,2}$. This is
equal to the sum of the lengths of the sides
of the quadrilateral on the axes, by Newton polygons.

Thus, by a mixed volume computation, we get
\[
    N_0(\eta_1,\eta_2) \le \frac 12 (\Per(a_{i,j}) - 2) + 
    (a_{1,1}-1) + (a_{1,2}-1) + (a_{2,1}-1) + (a_{2,2}-1)
    + 1,
\]
where the extra terms come from accounting
for zeros where one or both coordinates are zero, and the $\frac 12 (\Per(a_{i,j}) - 2)$ is the bound on the number
of non-degenerate zeros. 
Using the inequality $a_{i,j} \leq \mu_{2}k_{i,j}$, we get that this is at most
\[ 
    \frac{\mu_2^2}{2} (k_{1,1}k_{2,2} + k_{1,2}k_{2,1}) 
    + \mu_2(k_{1,1} + k_{1,2} + k_{2,1} + k_{2,2})
    - 4.
\]
Now, sum up the first term in the above over all the disks, and use the definition of $k_{i,j}$ to get
\begin{align*}
    \frac{\mu_2^2}{2}& \sum_{P \in \sm(K_i)}
    (\ord_P\omega_1 + 1)(\ord_P\omega_2 + 1) + (\ord_P\omega_1 + 1)(\ord_P\omega_2 + 1) \\
    & \le \mu_2^2(D_1 + (2g-1)^2).
\end{align*} 
Also, if we sum up the other terms, we get 
\begin{align*}
    \sum_{P \in \sm(K_i)} \mu_2(k_{1,1} + k_{1,2} + k_{2,1} + k_{2,2}) &= \mu_2\left(\sum_{P \in \sm(K_i)}
    2(\ord_P\omega_1 + \ord_P\omega_2) + 4\right) \\
    & \le 4\mu_2(2g-2) + 4\mu_2 D_1.
\end{align*}
Thus in total, there are at most 
\newsavebox\twoa

\savebox\twoa{$\displaystyle{
    \mu_2^2(D_1 + (2g-1)^2) + 4\mu_2 ((2g-2) + D_1) - 4
    D_1
}$}

\begin{equation}
    \label{case2abound}
    \usebox\twoa.
\end{equation}
rational points which reduce to $\{P, P\}$ for some
smooth point $P$.

\begin{remark}
Because the above bound is for one ramified extension, we should multiply the bound for $N_0(\eta_1,\eta_2)$ by $2$ to account for both ramified extensions. However, by the second half of
\autoref{setup}, there is a two-to-one correspondence
between zeros and points, so we must also multiply by $\frac{1}{2}$. 
\end{remark}

\subsection{Case 3(a)}
In this case $e = 1$.
Consider a residue polydisk $U_1\times U_2$ which is the preimage of $\{P_1, P_2\}$ for some \emph{distinct} smooth
points $P_1, P_2 \in \X_s(\mathbb{F}_p)$.
By the same reasoning as before, 
\[
	N_0(\eta_1,\eta_2)\le\frac 12 (\Per(a_{i,j}) - 2) + 
    (a_{1,1}-1) + (a_{1,2}-1) + (a_{2,1}-1) + (a_{2,2}-1)
    + 1.
\]
By \autoref{deltabound}, $a_{i,j} \le \mu_1 k_{i,j}$, so
the above is bounded by 
\[ 
    \frac{\mu_1^2}{2} (k_{1,1}k_{2,2} + k_{1,2}k_{2,1}) 
    + \mu_1(k_{1,1} + k_{1,2} + k_{2,1} + k_{2,2})
    - 4.
    \]
Note that $k_{i,j}$ depends on $P = \{P_1, P_2\}$, and 
\[ k_{i,j} = \ord_{P_j}\omega_i + 1. \]
We sum up the linear terms over the possible disks. In this case, the points reduce to arbitrary different disks so we must take a double sum, so we have
\begin{align*}
    \sum_{U_1\times U_2} N_0(\eta_1,\eta_2)&\le \sum_{P_1,P_2 \in \sm(\QQ_p)} \mu_1(k_{1,1} + k_{1,2} + k_{2,1} + k_{2,2}) \\
    &\qquad = \mu_1 \sum_{P_1,P_2 \in \sm(\QQ_p)} \ord_{P_1}\omega_1 + \ord_{P_2}\omega_1
    + \ord_{P_1}\omega_2 + \ord_{P_2}\omega_2 + 4 \\
    &\qquad = 2\mu_1 \left( \sum_{P_1,P_2 \in \sm(\QQ_p)} \ord_{P_1}\omega_1 + \ord_{P_2}\omega_1\right) + 4\mu_1\binom{D_1}{2}.
\end{align*}
Now notice that we have $D_1$ disks in $\sm(\mathbb{Q}_p)$ under consideration, 
and thus each $P_1 \in \sm(\mathbb{Q}_p)$, has $D_1 - 1$ 
possible paired $P_2 \in \sm(\mathbb{Q}_p)$. 
Thus, by \autoref{diskriemannroch} we can bound the above by
\begin{align*}
    \mu_1\left((4g-4)(D_1-1) + 4\binom{D_1}{2}\right).
\end{align*}
We bound the summation of the first term as followss.
\begin{align*}
    \frac{\mu_1^2}{2} &\sum_{P_1, P_2 \in \sm(\QQ_p)}
    (\ord_{P_1}\omega_1 + 1)(\ord_{P_2}\omega_2 + 1)
    + (\ord_{P_1}\omega_2 + 1)(\ord_{P_2}\omega_1 + 1)
    \\
    &\le \frac{\mu_1^2}{2} \left(\sum_{P_1 \in \sm(\QQ_p)}
    \ord_{P_1}\omega_1 + 1 \right) \cdot
    \left(\sum_{P_2 \in \sm(\QQ_p)}
    \ord_{P_2}\omega_2 + 1\right) \\
    &\le \frac{\mu_1^2}{2}(2g-2 + D_1)^2.
\end{align*}
Thus, the final bound in this case is
\newsavebox\threea

\savebox\threea{$\displaystyle{
\frac{\mu_1^2}{2}(2g-2 + D_1)^2 + \mu_1\left((4g-4)(D_1-1) + 4\binom{D_1}{2}\right) - 4\binom{D_1}{2}
}$}
\begin{equation}
    \label{case3abound}
    \usebox\threea.  
\end{equation}

\subsection{Case 1(b), 2(b), and 3(b)}
Consider a residue polyannulus $U_1\times U_2$ which is the preimage of $\{P_1, P_2\}$ for some $P_1, P_2$ which are both nodes on $\X_s$. In Case 1(b), $P_1, P_2 \in \X_s(\mathbb{F}_{p^2}) \setminus \X_s(\mathbb{F}_p)$ and $P_1 = P_2^{\sigma}$. In Case 2(b), $P_1, P_2 \in \X_s(\mathbb{F}_p)$ and $P_1 = P_2$. In Case 3(b), $P_1, P_2 \in \X_s(\mathbb{F}_p)$ and $P_1\neq P_2$. Also,
$e = 1$, $2$, $1$ in Cases 1(b), 2(b), and 3(b) respectively.

By \autoref{pureseries} and \autoref{krzb418}, 
if $P_1$ and $P_2$ are both nodes on $\X_s$,
then both $\eta_1$ and $\eta_2$ can be approximated
by Laurent polynomials whose Newton polygons are contained
in 
\[
    \conv(c_{1,1}e_1, c_{2,1}e_1, c_{1,2}e_2,  c_{2,2}e_2),
\]
for some $c_{i,j}$ with $c_{2,j} - c_{1,j} \le 2N_p(1/e, 2g-2) \le 8g-8$, according to \autoref{krzb418}.
It can be checked that the mixed volume of
two of these is at most $16(2g-2)^2$.
In the annulus case, there are no degenerate zeros because 
$0$ is not in the annulus.
Thus, for Cases 1, 2, 3 there are a total of
\[
   16(2g-2)^2\left(\frac12 \alpha + \alpha + \binom \alpha2\right) = 16(2g-2)^2\left(\frac32 \alpha + \binom \alpha2\right)
\]
zero-dimensional components.
So our final bound is
\newsavebox\onetwothreeb

\savebox\onetwothreeb{$\displaystyle{
16(2g-2)^2\left(\frac32 \alpha + \binom \alpha2\right)
}$}

\begin{equation}
    \label{case1b2b3bbound}
    \usebox\onetwothreeb.
\end{equation}


\subsection{Case 3(c)}
In this case $e = 1$.
Consider a residue polytube $U_1\times U_2$ 
which is the preimage of $\{P_1, P_2\}$ where $P_1, P_2 \in \X_s(\mathbb{F}_p)$, and $P_1$ is a smooth point on $\X_s$ and $P_2$ is a node.
In this case the Newton polygons can be bounded by
\[
    \conv(0, a_{1,1}e_1, a_{2,1}e_1, c_{1,2} e_2, c_{2,2}e_2)
\]
with $c_{1,2} - c_{2,2} \le 8g-8$, again by \protect{\cite[Corollary~4.18]{katz2015uniform}}.
It can be shown that the mixed volume of any two
of these is at most
\[
    (4g-4) (a_{1,1} + a_{2,1}).
\]
As $U_2$ is an annulus, the only degenerate zeros must be those for which the first component is $0$. Here
each $\eta_i$ is the sum of a power series and a Laurent series. By centering the disk $U_1$, we can ensure that the power series have no constant term, so degenerate zeros are common zeros of the two Laurent series. Therefore, again by \protect{\cite[Corollary~4.18]{katz2015uniform}}, we have at most $4(2g-2)$ degenerate zeros.
Thus, in total, we have a bound of
\[
    N_0(\eta_1,\eta_2) \le (4g-4) (a_{1,1} + a_{2,1}) + 4(2g-2).
\]
The sum over all such pairs $\{P_1, P_2\}$ is bounded by
\begin{align*}
    2(2g-2)&\sum_{P_1 \in \sm(\QQ_p)}\sum_{P_2 \in \nodes(\QQ_p)}
    \mu_1(2 + \ord_{P_1}\omega_1 + 
    \ord_{P_1}\omega_2) + 2 \\
    &\le (4g-4)(2\alpha\mu_1(2g-2)  + (2\mu_1 + 2)D_1\alpha).
\end{align*}
Thus, the bound in this case is
\newsavebox\threec
\savebox\threec{$\displaystyle{
(8g-8)(\alpha\mu_1(2g-2)  + (\mu_1 + 1)D_1\alpha)
}$}
\begin{equation}
    \label{case3cbound}
    \usebox\threec.
\end{equation}

To summarize our computations, we have the following table.
\renewcommand{\arraystretch}{2.14}
\begin{longtable}{| l | l | }\hline
    Case & Number of zero-dimensional components
\\*\hline\hline
    1(a) & \usebox\onea
\\\hline
    2(a) & \usebox\twoa
\\\hline
	3(a) & \usebox\threea
\\\hline 
    1(b), 2(b), 3(b) & \usebox\onetwothreeb
\\\hline
    3(c) & \usebox\threec
\\\hline
\caption{Data for zero-dimensional components}
\end{longtable}

\section{Uniform bounds}
\label{uniformbounds}

The above computations give us the following theorem by summing up over the various cases. Note that $d = 2$ and we assume $g \ge 4$, so \autoref{dlessthang} is automatically satisfied.
\bound*

\begin{proof}
	Let $p = 5$.
    By adding up Equations \ref{case1abound}, \ref{case2abound}, \ref{case3abound}, \ref{case1b2b3bbound} and \ref{case3cbound}, we get that there are at most
    \begin{gather*}
        128g^4 + 128g^3t + 32g^2t^2 + \frac{1616}3g^3 - \frac{1256}3g^2t - 344gt^2 - \frac{8858}9g^2 \\
        - 380gt + 662t^2 + \frac{11654}9g - 206t - \frac{4012}9
    \end{gather*}
    rational points.
    Since $0 \le t \le g$, we plug in $t = 0$
    for all the negative terms and $t = g$ for
    all of the positive terms in the above expression to get our statement.
\end{proof}


\subsection{Rank-favorable bounds for hyperelliptic curves}
Using \autoref{stoll77}, we can
compute rank-favorable bounds if $X$ is a hyperelliptic curve.
\begin{prop}[\protect{\cite[Proposition~7.7]{stoll2013uniform}}]
    \label{stoll77}
    Let $A_{\ol{P}}$ be an annulus in $X$ (the preimage of a node $\ol{P}$).
    Assume $p > e + 1$, where $e$ is the
    ramification index of $k$.
	Then for $\ell \ge 1$, any $g - \ell$ dimensional
    subspace of good differentials contains
    $\omega$ such that the number of zeros of $\eta_\omega$  in 
    $A(k^{\unr})$ is at most $2\mu_e \ell$, where
    $\mu_e$ is as in \autoref{deltabound}.
\end{prop}

	By \autoref{choosegoodomegas} we have
    a $g - r - 2$ dimensional space $V$ of good 	differentials.
    Applying \autoref{stoll77} with $\ell = r + 3$
    shows that any $g - r - 3$ dimensional subspace
    of $V$ contains $\omega$ such that $\eta_\omega$
    has at most $2\mu_e(r + 3)$ zeros.
    Thus, for each
	point $\{P_1, P_2\} \in (\Sym^2\X_s)(\FF_p)$, where
	$P_2$ is a node, we can choose linearly independent $\omega_1$
	and $\omega_2$ such that the number of zeros
	of $\eta_i$ on $A_{P_2}$ is at most $2\mu_e (r + 3)$.
    
\begin{remark}
	By \autoref{countmult}, this is equivalent to the statement 
    that there exists an interval $[c_1, c_2]$ of length at most
    $2\mu_e (r + 3)$ such that $u \notin [c_1, c_2]$ implies $u \notin
    \vrt_w(\eta_i)$ for $w \in (0, a)$, where $A_{\ol P}$ is the annulus given by
    \[ \{z \mid 0 < v(z) < a \}.\]
\end{remark}
Then by adjusting the bounds computed in Equations \ref{case1abound}, \ref{case2abound}, \ref{case3abound}, \ref{case1b2b3bbound}, \ref{case3cbound} using \autoref{stoll77}, we can generate rank-favorable uniform bounds on the number of quadratic points of hyperelliptic curves.

\hyperelliptic*

\begin{proof}
    Our proof follows the same lines as the proof of \autoref{finalboundcurves}. Specifically, the cases 
    for which the computations change are 1(b), 2(b), 3(b),
    and 3(c).
    For cases 1(b), 2(b), and 3(b), by \autoref{stoll77}, we can now say that our Newton polygon is contained within
    \[
        \conv(c_{1,1}e_1, c_{2,1}e_1, c_{1,2}e_2, c_{2,2}e_2)
    \]
    for some $c_{i,j}$ with $c_{2,1}-c_{1,1} \leq 8g-8$ and $c_{2,2}-c_{1,2} \leq 2\mu_er$. The mixed volume of two
    of these is at most $8(4g-4)(\mu_e(r+3)) \leq 8\mu_e(r + 3)(2g-2)$. Thus our bound for case 1(b), 2(b), 3(b) is
    \[
        8\mu_1(r+3)(2g-2)\left(\frac12 \alpha\right)
 + 
        8\mu_2(r+3)(2g-2)(\alpha)
+
        8\mu_1(r+3)(2g-2)\binom{\alpha}{2}.
    \]
    Now we adapt our bound of case 3(c). 
    We can choose differentials so that the Newton polygons are contained in
    \[
        \conv(0, a_{1,1}e_1, a_{2,1}e_1, c_{1,2} e_2, c_{2,2}e_2)
    \]
    where $c_{2,2} - c_{1,2} \le 2\mu_1 (r+3)$
    by \autoref{stoll77}.
    Then by a similar computation as before,
    we have at most    
    \[
        2\mu_1 (r+3)(\alpha\mu_1(2g-2) + (\mu_1 + 1)D_1\alpha)
    \]
    zero dimensional components.
    Then, as above we can plug in explicit values of $p$, $\mu_1$ and $\mu_2$ to get a bound
    \begin{gather*}
		\frac{128}{3}g^3r + \frac{128}{3}g^2tr + \frac{32}{3}gt^2r + \frac{128}g^3 + 128g^2t + 32gt^2 + \frac{2192}{9}g^2r - \frac{776}{9}gtr - 104t^2r \\ + \frac{69353}{36}g^2  - \frac{8807}{6}gt - \frac{23}{4}t^2 - \frac{5624}{9}gr + \frac{2072}{9}tr - \frac{99575}{36}g + \frac{4531}{4}t + \frac{736}{3}r + \frac{8053}{9}.
    \end{gather*}
    Using $0 \le t \le g$ as before gives an upper bound of
    \[
		96g^3r + 288g^3 + \frac{2192}{9}g^2r + \frac{69353}{36}g^2 - \frac{1184}{3}gr - \frac{14699}{9}g + \frac{736}{3}r + \frac{8053}{9}.
    \]
\end{proof}

\torsionpackets*
\begin{proof}
    Take the rank-favorable bound given in \autoref{hyperellipticrankfavorable}. If we let $r=0$, we immediately get the size of torsion packets. Thus our statement is proved.
\end{proof}

\bibliographystyle{amsalpha}
\def\bibfont{\small}
\bibliography{Master}{} 
\end{document}